\documentclass{cmslatex}
\usepackage{latexsym, amssymb, enumerate, amsmath}

\renewcommand{\theequation}{\arabic{section}.\arabic{equation}}

\newtheorem{thm}{Theorem}[section]
\newtheorem{prop}[thm]{Proposition}

\newtheorem{rem}[thm]{Remark}

\newcommand{\real}{\mathbb{R}}

\numberwithin{equation}{section}

\begin{document}

\title{Wild solutions for 2D incompressible ideal flow with passive tracer
\thanks{
}}
 \author{Anne C. Bronzi
 \thanks {Department of Mathematics, Statistics and Computer Science,
University of Illinois,
 Chicago, IL 60607, U.S.A., (annebronzi@gmail.com).} \and Milton C. Lopes Filho
 \thanks {Instituto de Matem\'atica, Universidade Federal do Rio de Janeiro,
 Cidade Universit\'aria, Ilha do Fund\~ao, Caixa Postal 68530,
 21941-909 Rio de Janeiro, RJ, Brasil, (mlopes@im.ufrj.br).} \and
 Helena J. Nussenzveig Lopes \thanks { Instituto de Matem\'atica,
 Universidade Federal do Rio de Janeiro, Cidade Universit\'aria,
 Ilha do Fund\~ao, Caixa Postal 68530, 21941-909 Rio de Janeiro, RJ, Brasil, (hlopes@im.ufrj.br).}}

\pagestyle{myheadings} \markboth{Wild solutions for 2D incompressible ideal
flow with passive tracer}{A. C.  Bronzi, M. C. Lopes Filho, H. J. Nussenzveig Lopes}
\maketitle

\begin{abstract} In \cite{delellis} C. De Lellis and L. Sz\'ekelyhidi Jr.
constructed wild solutions of the
incompressible Euler equations using a reformulation
of the Euler equations as a differential inclusion together with convex integration. In
this article we adapt their construction to the system consisting of adding the
transport of a passive scalar to the two-dimensional incompressible Euler equations.
\end{abstract}


\section{Introduction}

In this article, we present an adaptation of De Lellis and Sz\'ekelyhidi's
nonuniqueness construction \cite{delellis} to the system obtained by adding a
passive tracer equation to the two-dimensional incompressible Euler equations.
More precisely,
we are concerned with the system:
\begin{eqnarray}\label{p-t}
\left\{\begin{array}{l}
\partial_t v + (v\cdot\nabla)v+\nabla p=0\\
\partial_t b + (v\cdot\nabla)b=0\\
\mbox{\textit{div} }v=0,
\end{array}\right.
\end{eqnarray}
where $v=v(x,t)\in \mathbb R^2$ is the velocity field, $b=b(x,t)\in\mathbb R$ is
the tracer, $p=p(x,t)\in\mathbb R$ is the scalar pressure and $(x,t)\in \mathbb
R^2\times \mathbb R$.

In what follows, we will produce a velocity field $v \in L^{\infty}
(\real^2\times\real;\real^2)$ and a scalar $b\in L^{\infty}(\real^2\times\real;\real)$
who are compactly supported in space-time, such that both $v$ and $b$ are non-zero
in a set of positive measure in space-time, and such that $v$ is a weak solution
of the two-dimensional incompressible Euler equations while $b$ is a weak solution
of the linear transport equation with velocity $v$.

The initial motivation for the present work was to produce a wild solution for
the three-dimensional ideal incompressible magnetohydrodynamics (MHD) equations.
We have not managed to accomplish this and we will point out, later in this work,
the difficulty in achieving it. Nevertheless, we observe that, under a special
symmetry, the 3D MHD equations can be regarded as system \eqref{p-t}. Thus, our
construction may be interpreted as the existence of  wild solutions for the
(symmetry reduced) 3D MHD equations (see Section \ref{sec4}).

The construction of wild solutions has been extended to other problems,
namely for the incompressible porous media equations, see \cite{CFG} and
for a class of active scalar equations, see \cite{shvydkoy}. Our work is
yet another extension of De Lellis and Sz\'ekelyhidi's work; in fact, as
already mentioned, we will see that system \eqref{p-t} can also be interpreted
as a special case of the ideal incompressible MHD equations.

Let us now formulate more precisely the problem we address. We say that a
vector field $(v,b)=(v,b)(x,t)\in
L^2_{loc}(\mathbb{R}_x^2\times \mathbb{R}_t;\mathbb{R}^2\times
\mathbb R)$ is a \textit{weak solution} of (\ref{p-t}) if, for any test
function $\varphi = \varphi(x,t) \in \mathcal{C}^{\infty}_c(\mathbb{R}_x^2\times
\mathbb{R}_t;\mathbb{R})$ and any test vector field $\Psi=\Psi(x,t)\in
\mathcal{C}^{\infty}_c(\mathbb{R}_x^2\times \mathbb{R}_t;\mathbb{R}^2)$
such that $\mathrm{div} \ \Psi=0$, it holds that
\begin{eqnarray*}
&&\int \int (v\cdot \partial_t \Psi + (v\otimes v): \nabla
 \Psi) dxdt = 0, \\
&&\int \int (b \partial_t\varphi+(v b)\cdot \nabla \varphi)dxdt=0,\\
&&\int \int v\cdot \nabla \varphi dxdt = 0;
\end{eqnarray*}
above $v\otimes v$ is the $2\times 2$ matrix given by $(v\otimes v)_{ij}=v_iv_j$,
and $A:B$ stands for the Frobenius product of two $2\times 2$ matrices
$A=(a_{ij})$ and $B=(b_{ij})$, given by $(A : B)=\sum_{i,j=1}^2a_{ij}b_{ij}$.

Our goal in this paper is to construct a special class of weak solutions of
(\ref{p-t}) by using convex integration. We are going to follow the approach
presented in \cite{delellis}, where De Lellis and Sz\'ekelyhidi rewrite the Euler
equations as a differential inclusion and use convex integration to
construct weak solutions of the Euler equations with compact support in time and
space.

The main result of this paper is stated bellow.
\begin{theorem}\label{theo}
Given a bounded domain $\Omega\subset\mathbb R^2\times \mathbb R$,
there exists a weak solution $(v,b) \in L^{\infty}(\mathbb R_x^2\times
\mathbb R_t; \mathbb R^2\times \mathbb R)$ of the Euler equations with a
passive tracer (\ref{p-t}) such that
\begin{itemize}
\item [(i)] $|v(x,t)|=1$ and $|b(x,t)|=1$ for almost every $(x,t)\in \Omega$;
\item [(ii)] $v(x,t)=0$, $b(x,t)=0$ and $p(x,t)=0$  for almost every $(x,t)\in\mathbb
R^2\times \mathbb R\setminus \Omega$.
\end{itemize}
\end{theorem}
The remainder of this article is divided as follows: in Section \ref{sec1} we
place system \eqref{p-t} in the differential inclusion framework and we
provide the main ingredients to perform the convex integration scheme.
In Section \ref{sec3} we prove Theorem \ref{theo} and, finally, in
Section \ref{sec4} we apply the result to the ideal incompressible
MHD system and we add some concluding remarks.

\section{Convex Integration Scheme}\label{sec1}

Following the usual approach, we rewrite \eqref{p-t} as a system of
linear PDE's
\begin{eqnarray}\label{linear}
\left\{\begin{array}{l}
\partial_t v +\mbox{div }M +\nabla q=0\\
\mbox{div }v=0\\
\partial_t b +\mbox{div }w=0
\end{array}\right.
\end{eqnarray}
where
\begin{eqnarray}\label{constrain}
q=p+\dfrac{|v|^2}{2},\quad M=v\otimes v-\frac{1}{2}|v|^2I_2
\quad \mbox{and}\quad w=bv.
\end{eqnarray}
We define the \textit{constraint set}
$\mathcal K$ by $\mathcal K:= K\times [-1,1]$ with
\begin{equation*}
K=\left\{(b,w,v,M) \in
\{-1,1\}\times
S^1\times S^1\times\mathbb{S}_0^2 : M=v\otimes v-\displaystyle{\frac{1}{2}} I_2, w=bv\right\},
\end{equation*}
where $S^1$ denotes the one dimensional sphere and  $\mathbb S^2_0$ is the set
of symmetric $2\times 2$ matrices with vanishing trace.

It is clear that any solution $(b,w,v,M,q)$ of \eqref{linear} with image
contained in $\mathcal K$ is a solution of \eqref{p-t}.

We introduce the following $4\times 3$ matrix field
\begin{equation} \label{defineV}
V=\left( \begin{array}{cc}
M+qI_2 & v\\
v^t & 0\\
w^t & b
  \end{array}\right)
\end{equation}
and a new coordinate system $y=(x_1,x_2,t)\in \mathbb R^3$. In this setting,
equation (\ref{linear}) reduces to
\begin{eqnarray}\label{div}
\mbox{div}_y V=0.
\end{eqnarray}

Let $\mathcal M_{3\times 3}$ be the set of
symmetric $3\times3$ matrices $A$ such that $A_{3,3}=0$ and
$\mathcal M_{4\times 3}$ the set of
$4\times3$ matrices $A$ such that $(A_{i,j})_{i,j=1,2,3}\in \mathcal M_{3\times 3}$.
Observe that the following linear maps are
isomorphisms
\begin{eqnarray}
 \mathbb R^2\times\mathbb S^2_0\times \mathbb R &\longrightarrow & \mathcal M_{3\times 3}\\
\nonumber (v,M,q)&\longmapsto& \left( \begin{array}{cc}
M+qI_2 & v\\
v^t & 0
  \end{array}\right)
\end{eqnarray}
\begin{eqnarray}
\mathbb R\times\mathbb R^2 &\longrightarrow& \mathbb R^3\\ \nonumber
(b,w)&\longmapsto&\left( \begin{array}{cc}
w^t& b
  \end{array}\right).
\end{eqnarray}
\begin{eqnarray}
\mathbb R\times\mathbb R^2\times \mathbb R^2\times\mathbb S^2_0\times
\mathbb R &\longrightarrow & \mathcal M_{4\times 3}\\
\nonumber (b,w,v,M,q)&\longmapsto& \left( \begin{array}{cc}
M+qI_2 & v\\
v^t & 0\\
w^t& b
  \end{array}\right)
\end{eqnarray}

As the equivalent representations above are the most natural ones, we will,
from now on, not distinguish which one we will be considering, as
it should be clear from the context.

Recall that a \textit{plane wave} solution of \eqref{div} is a solution $V$,
as in \eqref{defineV}, of the form $V=V(y)=Uh(y\cdot
\xi)$, where $h: \mathbb R\rightarrow \mathbb R$ and where
$U\in \mathcal M_{4\times3}$. The \textit{wave cone} is then the set of states
of the planar solutions, that is, the set of states
$U\in \mathcal M_{4\times
3}$ such that
$V(y)=Uh(y\cdot \xi)$ is a solution of \eqref{div} for any $h$. In our case, the wave cone is given by
\begin{equation*}\Lambda=\left\{U\in  \mathcal M_{4\times3}:  \exists \xi \in \mathbb R^3\setminus \{0\}
\; \mbox{such that}\;U\xi =0 \right\},
\end{equation*}
or, equivalently,
\begin{equation*}\Lambda=\left\{(b,w,v,M,q)\in\mathbb{R}\times\mathbb{R}^2\times\mathbb{R}^2
\times\mathbb{S}_0^2\times\mathbb R:  \exists \xi \in \mathbb R^3\setminus \{0\}
 \mbox{ s.t. }\left( \begin{array}{cc}
M+qI_2 & v\\
v^t & 0\\
w^t& b
  \end{array}\right)\xi =0 \right\}.
\end{equation*}

We introduce the relaxed set
\begin{equation*}
\mathcal{U}=\mbox{int} (K^{co} \times [-1,1]),
\end{equation*}
where $K^{co}$ is the convex hull of $K$. One property which is important and
not hard to verify is that
$0\in \mathcal U$. The proof goes along the same line as the one presented
in \cite{delellis}.

We say that $(b,w,v,M,q)$ is a \textit{subsolution} of \eqref{p-t} if
$b\in L^2_{loc}(\mathbb R^2_x\times \mathbb R_t; \mathbb{R}), w,
v\in L^2_{loc}(\mathbb R^2_x\times \mathbb R_t;\mathbb{R}^2), M\in
L^1_{loc}(\mathbb R^2_x\times \mathbb R_t;\mathbb{S}_0^2)$ and $p$ is a distribution,
if $(b,w,v,M,q)$ is a solution of \eqref{linear} and if the image of
$(b,w,v,M,q)$ is contained in $\mathcal U$.

The idea of the convex integration scheme is to construct a sequence of
oscillating solutions which are obtained from adding localized versions of plane wave
solutions to subsolutions of \eqref{p-t}. In order to do so, it is important
to have the wave cone $\Lambda$ large enough so that it is possible to
construct oscillating solutions collinear to a suitable fixed direction.
In our case, it is easy to see  that, for all $b\in \mathbb R$, $v \in
\mathbb R^2$ and $M\in\mathbb S^2_0$, there  exist $q \in \mathbb R$ and
$w\in\mathbb R^2$ such that $(b,w,v,M,q)\in \Lambda$, which guarantees that the wave cone is large.

The following result provides the ''good'' directions to oscillate in the sense
that, by adding localized versions of plane waves
in these directions to a subsolution, one still obtains a subsolution.
\begin{lemma}\label{LI}
There exists  a constant $C>0$ such that, for each $(b,w,v,M,q)\in \mathcal{U}$,
there exists $(\bar{b},\bar{w},\bar{v},\bar{M})\in\mathbb{R}\times \mathbb
R^2\times\mathbb{R}^2\times\mathbb{S}_0^2$ satisfying \begin{itemize}
\item [(i)] $(\bar{b},\bar{w},\bar{v},\bar{M},0)\in\Lambda$;
\item [(ii)] the line segment with endpoints
$(b,w,v,M,q)\pm(\bar{b},\bar{w},\bar{v},
\bar{M},0)$ belongs to $\mathcal{U}$;
\item [(iii)] \begin{equation}\label{est}
|(\bar v,\bar b)|\geq C(2-(|v|^2+|b|^2)).
\end{equation}
\end{itemize}
\end{lemma}
\begin{proof}
Let $h=(b,w,v,M)\in \mbox{int}K^{co}$. By Carath\'eodory's theorem there
exist $\lambda_i \in (0,1)$, with $\sum_{i=1}^{N+1}
\lambda_i=1$ and $h_i=(b_i,w_i,v_i,M_i)\in
K$, for $i=1,\ldots, N+1$, such that
\begin{equation*}
h=\sum_{i=1}^{N+1}\lambda_ih_i,
\end{equation*}
where $N=7$ is the dimension of $\mathbb{R}\times\mathbb{R}^2\times\mathbb{R}^2\times\mathbb{S}_0^2$.

Suppose that $\lambda_1=\max \lambda_i$ and define $i^*$ the index such that
\begin{equation*}\lambda_{i^*}^2(|v_{i^*}-v_1|^2+|b_{i^*}-b_1|^2)=\max\{
\lambda_i^2(|v_i-v_1|^2+|b_i-b_1|^2): i=1,\ldots,8\}.
\end{equation*}

Observe that, since $h=\sum_{i=1}^8\lambda_{i}h_i$ and  $h-h_1=\sum_{i=2}^8
\lambda_{i}(h_i-h_1)$, we obtain
\begin{multline*}
|(v-v_1,b-b_1)|=\left|\sum_{i=2}^8\lambda_{i}(v_i-v_1,b_i-b_1)\right|\leq\sum_
{ i=2}^8\sqrt{\lambda_i^2
(|v_i-v_1|^2+|b_i-b_1|^2)}\\\leq
7\sqrt{\max\{\lambda_i^2(|v_i-v_1|^2+|b_i-b_1|^2):
i=1,\ldots,8\}}=7\sqrt{\lambda_{i^*}^2(|v_{i^*}-v_1|^2+|b_{i^*}-b_1|^2)}.
\end{multline*}
Thus, $|v-v_1|^2+|b-b_1|^2\leq 49\lambda_{i^*}^2(|v_{i^*}-v_1|^2+|b_{i^*}-b_1|^2).$
Then, it follows  by defining $\bar
h:=\frac{1}{2}\lambda_{i^*}(h_{i^*}-h_1)$ that
\begin{multline*}
\frac{1}{28\sqrt{2}}(2-(|v|^2+|b|^2))\leq\frac{1}{14\sqrt{2}}
((1-|v|)+(1-|b|))\leq\frac{1}{14\sqrt{2}}\sqrt{2}
\sqrt{(1-|v|)^2+(1-|b|)^2}
\\\leq\frac{1}{14}\sqrt{|v-v_1|^2+|b-b_1|^2}\leq
\frac{1}{14}7\lambda_{i^*}\sqrt{|v_{i^*}-v_1|^2+|b_{i^*}-b_1|^2}=
|(\bar v, \bar b)|.
\end{multline*}

Therefore, $|(\bar v, \bar b)|\geq C(2-(|v|^2+|b|^2))$.

It is easy to see that $(\bar
h,0)=(\bar{b},\bar{w},\bar{v},\bar{M},0)\in\Lambda$. Indeed, write
$v_{i^*}=(v_{i^*}^1, v_{i^*}^2)$ and $v_1=(v_1^1, v_1^2)$ and set
\begin{equation*}
\xi=\left(-\frac{v_{i^*}^2-v_1^2}{v_{i^*}^1-v_1^1},1,
-\frac{v_{i^*}^1v_{1}^2-v_1^1v_{i^*}^2}{v_{i^*}^1-v_1^1}\right).
\end{equation*}

Then
\begin{equation*}\left[\begin{array}{lll}
\bar M+ 0 I_2& \bar v\\
\bar v^t & 0\\
\bar w& \bar b
\end{array}
\right]\xi=
\frac{1}{2}\lambda_{i^*}\left[\begin{array}{lll}
(v_{i^*}^1)^2-(v_{1}^1)^2 &
v_{i^*}^1v_{i^*}^2-v_{1}^1v_{1}^2&v_{i^*}^1-v_{1}^1\\
v_{i^*}^2v_{i^*}^1-v_{1}^2v_{1}^1 &
(v_{i^*}^2)^2-(v_{1}^2)^2&v_{i^*}^2-v_{1}^2\\
v_{i^*}^1-v_{1}^1 & v_{i^*}^2-v_{1}^2& 0\\
b_{i^*}v_{i^*}^1-b_1v_{1}^1 & b_{i^*}v_{i^*}^2-b_1v_{1}^2 & b_{i^*}-b_1
             \end{array}
\right]\xi =0.
\end{equation*}
\end{proof}

\begin{rem}
In the proof of Lemma \ref{LI} we showed that, if $z_1,z_2 \in K$, then
$z_1-z_2\in \Lambda$. Thus, the $\Lambda$-convex hull of $K$ coincides with the
convex hull of $K$ (see \cite{delellis12} for more details). In contrast,
for the 3D MHD system case we were not able to prove this type of result and,
consequently, we could not perform the convex integration.
\end{rem}

It is clear that the only compactly supported plane wave is the trivial one.
Although we cannot work with an exact wave solution, in the
next result we construct a plane wave-like solution, which is a compactly
supported solution of \eqref{div} living in a small neighborhood of the
line spanned by a fixed wave state.
\begin{prop}\label{prop2}Let $\bar V\in\Lambda$ be such that $\bar V e_3\neq 0$.
Let $\sigma$ be the line joining the points $-\bar V$ and $\bar V$ in
$\mathcal M_{4\times 3}$. Then, there exists $\alpha>0$ such that, for every
$\varepsilon>0$, there exists a smooth $4\times 3$ matrix field $V$ given by
\[ V(x,t)=\left( \begin{array}{cc}
M(x,t)+q(x,t)I_2 & v(x,t)\\
v(x,t)^t & 0\\
w(x,t)^t & b(x,t)
  \end{array}\right),\]
where $M\in \mathbb{S}_0^2$, $v,w\in \mathbb R^2$, $b\in \mathbb R$, satisfying
the following properties:
\begin{itemize}
\item [(i)] $\mbox{div}_{(x,t)}V=0$;
\item [(ii)] $\mbox{supp }V\subset B_1(0)$;
\item [(iii)] $Im \;V\subset \sigma_\varepsilon=\{A\in \mathcal M_{4\times 3}:
\mbox{dist }(A, \sigma)<\varepsilon\}$;
\item [(iv)] $\int|v(y)|dy\geq \alpha|\bar v|$ and $\int|b(y)|dy\geq
\alpha|\bar b|$, where $\bar v=(\bar V_{i,3})_{i=1,2}$ and $\bar b=\bar V_{4,3}$.
\end{itemize}
\end{prop}

\begin{proof}Let us write $\bar V=\left(\begin{array}{l}\bar U\\\bar W^t\end{array}\right)$,
where $\bar U=\left(\begin{array}{cc}\bar M+\bar qI_2& \bar v\\\bar v^t& 0\end{array}\right)$
and $\bar W=(\bar w,\bar b)$. Observe that  $\bar U$ is exactly the matrix arising in
the differential inclusion associated to the incompressible Euler equations,
see \cite{delellis}. Therefore, we can use \cite[Proposition 3.2]{delellis}
to obtain the existence of a matrix field $U: \mathbb R^2_x\times \mathbb R_t
\rightarrow \mathcal M_{3\times 3}$ such that $\mbox{div}_{(x,t)}U=0$,
$\mbox{supp }U\subset B_1(0)$, $Im \;U\subset \{A\in \mathcal M_{3\times 3}:
\mbox{dist }(A, \sigma_{\bar U})<\varepsilon\}$, where $\sigma_{\bar U}$ is
the line joining the points $-\bar U$ and $\bar U$ in $\mathcal M_{3\times 3}$,
and $\int|Ue_3(y)|dy\geq \alpha|\bar v|$.

Now, we will construct $W: \mathbb R^2_x\times \mathbb R_t\rightarrow
\mathbb R^3$ such that $\mbox{div}_{(x,t)}W=0$,  $\mbox{supp }W\subset B_1(0)$,
$Im \;W\subset \{a\in\mathbb R^3: \mbox{dist }(a, \sigma_{\bar W})<\varepsilon\}$,
where $\sigma_{\bar W}$ is the line joining the points $-\bar W$ and $\bar W$
in $\mathbb R^3$, and $\int|W(y)\cdot e_3|dy\geq
\alpha|\bar b|$. Once we have done this we define $V=\left(\begin{array}{l}
U\\W^t\end{array}\right)$ so that it is clear that $V$ satisfies conditions
$(i)$ to $(iv)$ and the proposition is proved.

In order to do so we  divide the construction in two parts. First, we suppose that
$\bar W=(0,\bar w_2,\bar b)$ with $\bar b\neq 0$. Let $\phi: \mathbb R^3\rightarrow \mathbb
R$ be a smooth cutoff function such that $|\phi|\leq 1$, $\phi=1$
on $B_{1/2}(0)$ and $\mbox{supp}\;(\phi)\subset B_{1}(0)$.
Define
\begin{equation*}W(y)=\frac{1}{N^2}\left(\begin{array}{c}
\partial^2_{12}(\bar w_2\phi\sin(Ny_1))+\partial^2_{13}(\bar b\phi\sin(Ny_1))\\
-\partial^2_{11}(\bar w_2\phi\sin(Ny_1))\\
-\partial^2_{11}(\bar b\phi\sin(Ny_1))
\end{array}\right).
\end{equation*}

Note that $W$ is a smooth divergence-free vector field with support
contained in $B_1(0)$. Moreover, for $y\in B_{1/2}(0)$ we have that $W(y)=\bar W
\sin(Ny_1)$, thus \[\int |W(y)\cdot e_3|dy \geq \int_{B_{1/2}(0)} |W(y)\cdot e_3|dy=|\bar
W\cdot e_3| \int_{B_{1/2}(0)} |\sin(Ny_1)|dy\geq 2\alpha |\bar b|,\]
for some $\alpha >0$.

Define $\tilde W=\left(\begin{array}{c}
0\\
\bar w_2\sin(Ny_1)\\
\bar b\sin(Ny_1)
\end{array}\right)$ and observe that $||W-\phi \tilde W||_{\infty}\leq \frac{C}{N^2}||\phi||_{C^2}.$

Therefore, by taking $N$ sufficiently large we have that $||W-\phi \tilde W||_{\infty}<
\varepsilon$.

Finally, since $|\phi|\leq 1$ and $\tilde W$ takes value in $\sigma_{\bar W}$
the image of $\phi \tilde W$ is contained in $\sigma_{\bar W}$. Thus the image
of $W$ is contained
in the $\varepsilon$-neighborhood of $\sigma_{\bar W}$.

Next we consider the general case. By hypothesis $\bar W^t e_3\neq 0$ and
$\bar W^t \xi=0$ for some $\xi \in \mathbb R^3 \setminus\{0\}$. Clearly,
$\xi$ and $e_3$ are linearly independent. Set $\eta \in \mathbb R^3\setminus
\{0\}$ in such way that $\{\xi,\eta,e_3\}$ is a basis of $\mathbb R^3$ and let
$A$ be the $3\times 3$ matrix given by $Ae_1=\xi$, $Ae_2=\eta$ and $Ae_3=e_3$.

Define $\bar B=A^t\bar W$. It is clear that  $\bar B\in \mathbb R^3$,
$\bar B_1=0$ and $\bar B_3\neq 0$. Thus we use the above argument to construct a smooth divergence-free map
$B: \mathbb R^3\rightarrow \mathbb R^3$ with compact support
in $B_1(0)$ and image lying in the $||A||^{-1}\varepsilon$-neighborhood of the
line segment $\tau$ with endpoints $-\bar B$ and $\bar B$.

Set $W(y)=(A^{-1})^tB(A^ty)$. First observe that $W$ is supported in
$(A^{-1})^tB_1(0)$. Since the isomorphism $T: X \mapsto (A^{-1})^t X$ maps
$\tau$ into $\sigma_{\bar W}$, we have that the image of $W$ is contained in the
$\varepsilon$-neighborhood of $\sigma_{\bar W}$. The following straightforward
calculation shows that $W$ is divergence-free:
\begin{multline*}\int W(y)\cdot \nabla \phi(y)dy=\int((A^{-1})^tB(z))\cdot
\nabla \phi((A^{-1})^tz)(\det A)^{-1}dz\\=(\det A)^{-1}\int B(z)\cdot \nabla
(\phi((A^{-1})^tz))dz=0, \mbox{ for all }\phi\in \mathcal
C^\infty_c(\mathbb R^3;\mathbb R).
\end{multline*}

Finally, we have
\begin{eqnarray*}\int_{(A^{-1})^tB_1(0)}|W(y)^te_3|dy=\int_{(A^{-1})^tB_1(0)}|((A^{-1})^t
B(A^ty))^te_3|dy\\=\int_{B_1(0)}|((A^{-1})^tB(z))^te_3| \dfrac{dz}{|\det
A|}\geq \dfrac{2\alpha|((A^{-1})^t\bar B)^te_3|}{|\det A|}=\dfrac{2\alpha|(\bar
W)^te_3|}{|\det A|}.\end{eqnarray*}

To conclude, we observe that using the same argument of covering and rescaling
as the one presented in the proof of \cite[Proposition 3.2]{delellis} one can
rescale $W$ in such way that all desired properties remain valid.
\end{proof}

Let $X_0$ be the set of vector fields $(b,w,v,M,q) \in \mathcal
C^\infty(\mathbb R^2\times\mathbb R)$ that satisfy $(i),
(ii)$ and $(iii)$ below:
\begin{itemize}
\item [(i)] $\mbox{supp}(b,w,v,M,q)\subset \Omega$,
\item [(ii)] $(b,w,v,M,q)$ solves (\ref{linear}) in $\mathbb
R^2\times\mathbb R$,
\item [(iii)] $(b,w,v,M,q)(x,t) \in \mathcal{U}$ for all $(x,t)\in \mathbb
R^2\times\mathbb R$.
\end{itemize}

We endow $X_0$ with the topology of $L^{\infty}$-weak* convergence and we define
$X$ as the closure of $X_0$ in this topology.

It is straightforward that, if $(b,w,v,M,q) \in X$ is such that $|v(x,t)|=1,\;
|b(x,t)|=1$ for almost every $(x,t)\in \Omega$, then  $v$, $b$ and
$p:=q-\frac{1}{2}|v|^2$ are a weak solution of (\ref{p-t}) such that
$v(x,t)=0,\;b(x,t)=0$ and $p(x,t)=0$ for
all $(x,t) \in \mathbb R^2\times \mathbb R\setminus \Omega$.

Next we prove a key result to implement the convex integration scheme in the
proof of Theorem \ref{theo}.
\begin{lemma}\label{lemma4}
There exists a constant $\beta>0$ such that, for each $(b_0,w_0,v_0,M_0,q_0) \in X_0$,
there exists a sequence $(b_k,w_k,v_k,M_k,q_k)
\in X_0$ such that
\begin{eqnarray*}
&&\|v_k\|_{L^2(\Omega)}^2+\|b_k\|_{L^2(\Omega)}^2 \geq
\|v_0\|_{L^2(\Omega)}^2+\|b_0\|_{L^2(\Omega)}^2+
\beta(2|\Omega|-(\|v_0\|^2_{L^2(\Omega)}+\|b_0\|^2_{L^2(\Omega)}))^2,\\
&&\mbox{and }\quad (b_k,w_k,v_k,M_k,q_k)\stackrel{*}{\rightharpoonup}(b_0,w_0,v_0,M_0,q_0)
\quad \mbox{in} \quad L^{\infty}.
\end{eqnarray*}
\end{lemma}
Although the proof of Lemma \ref{lemma4} is a  simple adaptation of the proof of
Lemma 4.6 of \cite{delellis}, due its crucial role in this paper we are going
to reproduce the main steps of the proof.
\begin{proof}
Let $z_0:=(b_0,w_0,v_0,M_0,q_0)\in X_0$. We apply Lemma \ref{LI} to each
element of the compact set $\mbox{Im} (z_0)\subset \mathcal U$ so
that we obtain that, for each $(x,t)\in \Omega$, there exists a
direction
\begin{equation*}
\bar z(x,t):=(\bar b,\bar w,\bar v,\bar M,0)(x,t) \in \Lambda
\end{equation*}
such that the line segment with endpoints $z_0(x,t)\pm \bar z(x,t)$ is contained in $\mathcal{U}$, and
\begin{equation*}
|\bar v(x,t)|+|\bar b(x,t)|\geq \sqrt{|\bar v(x,t)|^2+|\bar b(x,t)|^2}
\geq C(2-(|v_0(x,t)|^2+|b_0(x,t)|^2)).
\end{equation*}
Observe that since $z_0\in X_0$ then $|\bar v(x,t)|+|\bar b(x,t)|>0$, for all $(x,t)\in \mathbb R^2\times \mathbb R$.
Also, it is clear from the construction of $(\bar b,\bar w,\bar v,\bar M)$ and the fact that
$(b_0,w_0,v_0,M_0,q_0)$ is uniformly continuous that there exists $\varepsilon>0$
such that, for any $(x,t),\;(x_0,t_0)\in \Omega$ with
$|x-x_0|+|t-t_0|<\varepsilon$, the $\varepsilon$-neighborhood of the line
segment with endpoints $z_0(x,t)\pm \bar z(x_0,t_0)$ is also contained in $\mathcal{U}$.

Now, since $\bar b\neq 0$ and $\bar v\neq 0$ we can use Proposition \ref{prop2}
with $(\bar b,\bar w,\bar v,\bar M,0)(x_0,t_0)\in \Lambda$ and $\varepsilon>0$ to obtain a
smooth solution $(b,w,v,M,q)$ of (\ref{linear}) with the properties stated in
the Proposition \ref{prop2}. For every $r<\varepsilon$ let
\begin{equation*}(b_r,w_r,v_r,M_r,q_r)(x,t)=(b,w,v,M,q)\left(\frac{x-x_0}{r},\frac{t-t_0}{r}
\right).
\end{equation*}
Therefore, $(b_r,w_r,v_r,M_r,q_r)$ is also a smooth solution of \eqref{linear},
satisfying the following properties
\begin{itemize}
\item [(i)] $\mbox{supp} ((b_r,w_r,v_r,M_r,q_r))\subset B_r(x_0,t_0)$,
\item [(ii)] $\mbox{Im} ((b_r,w_r,v_r,M_r,q_r))$ is contained in the $\varepsilon$-
neighborhood of the line segment with endpoints $\pm (\bar b,\bar w,\bar v,\bar M,0)(x_0,t_0)$,
\item [(iii)] $\int|v_r|dxdt\geq \alpha |\bar v(x_0,t_0)||B_r(x_0,t_0)|$ and
$\int|b_r|dxdt\geq \alpha |\bar b(x_0,t_0)||B_r(x_0,t_0)|.$
 \end{itemize}
It is clear from the above properties and the fact that the  line
segment with endpoints $z_0(x,t)\pm \bar z(x_0,t_0)$ is contained in $\mathcal{U}$
that, for any $r<\varepsilon$, we have $(b_0,w_0,v_0,M_0,q_0)+ (b_r,w_r,v_r,M_r,q_r)\in X_0.$

Now, since $v_0$ is uniformly continuous, one can find a radius $r_0>0$ such that,
for any $r<r_0$, there exists a finite family of
pairwise disjoint balls $B_{r_j}(x_j,t_j) \subset\Omega$ with $r_j<r$ such that
\begin{multline}\label{int1}
\int_{\Omega}(2-(|v_0(x,t)|^2+|b_0(x,t)|^2))dxdt\\
\leq2\sum_j (2-(|v_0(x_j,t_j)|^2+|b_0(x_j,t_j)|^2))|B_{r_j}(x_j,t_j)|.
\end{multline}

Next we fix $k \in \mathbb N$ such that $\frac{1}{k}<\min\{r_0,\varepsilon\}$
and we choose a  finite family of pairwise disjoint balls $B_{r_{k,j}}(x_{k,j},t_{k,j}) \subset
\Omega$ with radii $r_{k,j}<\frac{1}{k}$ for which \eqref{int1} is valid. In each
ball $B_{r_{k,j}}(x_{k,j},t_{k,j})$ we apply the construction above so that we
obtain a sequence of smooth solution of \eqref{linear}, namely
$(b_{k,j},w_{k,j},v_{k,j},M_{k,j},q_{k,j})$, satisfying the appropriated
versions of properties $(i),(ii)$ and $(iii)$. In particular, we have that
\begin{eqnarray*}
&&(b_k,w_k,v_k,M_k,q_k):=(b_0,w_0,v_0,M_0,q_0)+\sum_j(b_{k,j},w_{k,j},v_{k,j},
M_{k,j},q_{k,j}) \in X_0
\end{eqnarray*}
and, by property $(iii)$ and \eqref{int1},
\begin{equation}\label{34}
\int(|v_k-v_0|+|b_k-b_0|)dxdt\\
\geq \frac{C\alpha}{2} \int_{\Omega} (2-(|v_0|^2+|b_0|^2)dxdt.
\end{equation}

Finally, observe that $(b_k,w_k,v_k,M_k,q_k)\stackrel{*}{\rightharpoonup}
(b_0,w_0,v_0,M_0,q_0)$ in $L^{\infty}$. Consequently,
\begin{multline}\label{36}
\liminf_{k\rightarrow \infty}(\|v_k\|^2_{L^2(\Omega)}+\|b_k\|^2_{L^2(\Omega)})
=\|v_0\|^2_{L^2(\Omega)}+
\liminf_{k\rightarrow \infty}(2\langle v_0, v_k-v_0\rangle+\|v_k-v_0\|^2_{L^2(\Omega)})\\+
\|b_0\|^2_{L^2(\Omega)}+ \liminf_{k\rightarrow \infty}(2\langle b_0, b_k-b_0\rangle+\|b_k-b_0\|^2_{L^2(\Omega)})\\
\geq\|v_0\|_{L^2(\Omega)}^2+\|b_0\|_{L^2(\Omega)}^2+\frac{1}{|\Omega|}\liminf_{k\rightarrow \infty}
(\|v_k-v_0\|_{L^1(\Omega)}+\|b_k-b_0\|_{L^1(\Omega)})^2.
\end{multline}

In view of (\ref{34}) and (\ref{36}) we get
\begin{equation*}
\liminf_{k\rightarrow \infty}(\|v_k\|_{L^2}^2+\|b_k\|_{L^2}^2)\geq\|v_0\|_{L^2}^2
+\|b_0\|_{L^2}^2+ \frac{C^2\alpha
^2}{4|\Omega|}(2|\Omega|-(\|v_0\|^2_{L^2}+\|b_0\|^2_{L^2}))^2.
\end{equation*}
Thus we proved the lemma  with $\beta=\frac{C^2\alpha^2}{4|\Omega|}$.
\end{proof}


\section{Proof of the main theorem}\label{sec3}

\textit{Proof of Theorem \ref{theo}:} The idea of the proof is to construct a
sequence $\{z_k=(b_k,w_k,v_k,M_k,q_k)\}\subset
X_0$ satisfying the following conditions:
\begin{itemize}
\item [(i)] there exists $z=(b,w,v,M,q)\in X$ such that $z_k \rightarrow z$
strongly in $L^2(\mathbb R_x^2\times\mathbb R_t)$;
\item [(ii)]$\|v_{k+1}\|_{L^2}^2 +\|b_{k+1}\|_{L^2}^2\geq
\|v_k\|_{L^2}^2+\|b_k\|_{L^2}^2+
\beta(2|\Omega|-(\|v_k\|^2_{L^2}+\|b_{k}\|^2_{L^2}))^2$,
for all $ k\in \mathbb N$.
\end{itemize}

Then, using $(i)$ we can pass to the limit in $(ii)$ in order to obtain
\begin{eqnarray*}\|v\|_{L^2(\Omega)}^2+\|b\|_{L^2(\Omega)}^2 \geq \|v\|_{L^2(\Omega)}^2+
\|b\|_{L^2(\Omega)}^2+
\beta(2|\Omega|-(\|v\|^2_{L^2(\Omega)}+\|b\|^2_{L^2(\Omega)}))^2
\end{eqnarray*}
and hence $\|v\|^2_{L^2(\Omega)}+\|b\|^2_{L^2(\Omega)}=2|\Omega|$. Since
$|v|\leq1$, $|b|\leq 1$ in $\Omega$ and since they are supported in $\Omega$, we
conclude that $|v|=1_{\Omega}=|b|$. Clearly $(b,w,v,M)\in K^{co}$ for a.e.
$(x,t) \in \Omega$ since $(b,w,v,M,q)\in X$. This implies that $(b,w,v,M)(x,t)\in K$
for a.e. $(x,t)\in\Omega$, as we wished.

It remains to construct a sequence $\{z_k\}\in X_0$ satisfying $(i)$ and
$(ii)$. In order to do so, set $(b_1,w_1,v_1,M_1,q_1)\equiv 0$ in $\mathbb
R^2\times\mathbb R$ and let $\rho_{\varepsilon}$ be a standard mollifying kernel
in $\mathbb R^2_x\times\mathbb R_t$. The sequence
$(b_k,w_k,v_k,M_k,q_k)\in X_0$ is  constructed  inductively, as well as an
auxiliary sequence of numbers $\eta_k>0$, as described below.
Once we have obtained $z_j:=(b_j,w_j,v_j,M_j,q_j)$ for $j\leq k$ and
$\eta_1,\ldots, \eta_{k-1}$, we choose
\begin{equation}\label{1}
\eta_k<2^{-k}
\end{equation}
in such way that
\begin{eqnarray}\label{2}
\|z_k-z_k\ast\rho_{\eta_k}\|_{L^2(\Omega)}<2^{-k}.
\end{eqnarray}

We then apply Lemma \ref{lemma4} to obtain $z_{k+1}=(b_{k+1},w_{k+1},v_{k+1},
M_{k+1},q_{k+1})\in X_0$ such that
\begin{multline}\label{in}
\|v_{k+1}\|_{L^2(\Omega)}^2 +\|b_{k+1}\|_{L^2(\Omega)}^2\\\geq
\|v_k\|_{L^2(\Omega)}^2+\|b_k\|_{L^2(\Omega)}^2+
\beta(2|\Omega|-(\|v_k\|^2_{L^2(\Omega)}+\|b_{k}\|^2_{L^2(\Omega)}))^2
\end{multline}
\begin{eqnarray}\label{3}
\mbox{ and }\|(z_{k+1}-z_k)\ast\rho_{\eta_j}\|_{L^2(\Omega)}<2^{-k} \quad
\mbox{for all} \quad j\leq k.
\end{eqnarray}
Since the sequence $\{z_k\}$ is bounded in $L^{\infty}(\mathbb R^2_x\times
\mathbb R_t)$, there exists a subsequence, which we still denote by $z_k$, and a
vector field $z=(b,w,v,M,q)\in X$ such that $z_k \stackrel{*}{\rightharpoonup} z
$ in $L^{\infty} (\mathbb R^2_x\times\mathbb R_t)$. Moreover, the sequence
$\{z_k\}$ and the corresponding sequence $\{\eta_k\}$ satisfy the properties
\eqref{1}, \eqref{2}, \eqref{in} and \eqref{3}. Then, for every $k\in\mathbb N$
\begin{eqnarray*}
\|z_k\ast\rho_{\eta_k}-z\ast\rho_{\eta_k}\|_{L^2(\Omega)}\leq
\sum_{j=0}^{\infty}\|z_{k+j}\ast\rho_{\eta_k}-z_{k+j+1}\ast\rho_{\eta_k}\|_{
L^2(\Omega)}\leq\sum_{j=0}^{\infty}2^{-(k+j)}\leq 2^{-k+1},
\end{eqnarray*}
and since
\begin{eqnarray*}
\|z_k-z\|_{L^2(\Omega)}\leq\|z_k-z_k\ast\rho_{\eta_k}\|_{L^2(\Omega)}+
\|z_k\ast\rho_{\eta_k}-z\ast\rho_{\eta_k}\|_{L^2(\Omega)}+\|z\ast\rho_{\eta_k}
-z\|_{L^2(\Omega)},
\end{eqnarray*}
we deduce that $z_k\rightarrow z$ strongly in $L^2(\Omega)$. This concludes the
proof. \hfill $\Box$

We point out that the proof of the main theorem could be done using a Baire
category argument but we preferred to use an approximating procedure, which is
more constructive and could be useful for doing a numerical visualization,
along the lines of \cite{BLL}.

\section{Application to the MHD equations and concluding remarks}\label{sec4}

Consider the 3D magnetohydrodynamics equations (MHD),
\begin{eqnarray}
\left\{\begin{array}{l}\label{mhd3}
\partial_t u + (u\cdot\nabla)u-(\mbox{\textit{curl} }b)\times b +\nabla p=0\\
\partial_t b - \mbox{\textit{curl} }(u\times b)=0\\
\mbox{\textit{div} }u=0\\
\mbox{\textit{div} }b=0
\end{array}\right.
\end{eqnarray}where $u=u(x,t)\in \mathbb R^3$ is the velocity,
$b=b(x,t)\in\mathbb R^3$ is the magnetic induction, $p=p(x,t)\in\mathbb R$ is
the pressure and $(x,t)\in \mathbb R^3\times \mathbb R$.

The system (\ref{mhd3}) describes the motion of an ideal incompressible
conducting fluid interacting with a magnetic field. Note that by taking $b=0$ in
(\ref{mhd3}) we obtain the Euler equations.

Observe that, if we restrict ourself to the class of
solutions that preserve the following symmetry:
\begin{eqnarray*}
&&u=u(x,t)=(u_1(x,t),u_2(x,t),0),\\
&&b=b(x,t)=(0,0,b(x,t)),\\
&&(x,t)=(x_1,x_2,t)\in\mathbb R^2\times \mathbb R
\end{eqnarray*}
then system (\ref{mhd3}) is reduced to
\begin{eqnarray}
\left\{\begin{array}{l}\label{mhd2}
\partial_t u+ (u\cdot\nabla)u+\nabla \left(p+\dfrac{|b|^2}{2}\right)=0\\
\partial_t b + (u\cdot\nabla)b=0\\
\mbox{\textit{div} }u=0.
\end{array}\right.
\end{eqnarray}

System \eqref{mhd2} can be rewritten as the incompressible Euler equations
with a passive tracer by defining $\bar p=p+|b|^2/2$. Therefore, by Theorem
\ref{theo}, we can conclude that weak solutions for these equations are not
unique. In particular, solutions of (\ref{mhd2}) are solutions of the full
3D MHD, so that weak solutions of the 3D MHD are not unique.

We add some concluding remarks. The work of De Lellis and Sz\'ekelyhidi has
generated substantial ongoing activity concerning weak solutions of the
incompressible Euler equations, mainly along the following direction: the
construction of dissipative solutions together with improving the regularity
of the velocity field in wild solutions, see
\cite{B,BLS,CLS,delellis2,delellis13,Isett13} and references therein.  This
line of investigation naturally suggests a direction for future research,
that of constructing dissipative solutions of system \eqref{p-t} and seeking
wild solutions of \eqref{p-t} with improved regularity. In addition, it would
be interesting to find a broader class of examples of this construction for
the ideal MHD system and to provide a computational visualization of these
solutions along the lines of what was done in \cite{BLL} for Shnirelman's
example. In fact, it is more natural to construct such a visualization for
the passive-tracer example, than it would be for the original De Lellis
and Sz\'ekelyhidi construction.

\medskip
\footnotesize
{\bf Acknowledgments:} Research of Helena J. Nussenzveig Lopes is supported in part
by CNPq grant \# 306331/2010-1, and by FAPERJ grant \# E-26/103.197/2012.
Research of Milton C. Lopes Filho is supported in part by CNPq grant
\# 303089/2010-5. Anne C. Bronzi's research is supported by Post-Doctoral grant
\# 236994/2012-3. This work was supported by FAPESP Thematic Project \# 2007/51490-7
and by the FAPESP grant \# 05/58136-9. The research presented here was part
of Anne C. Bronzi's doctoral dissertation at the
Mathematics graduate program of UNICAMP. The authors thank C. De Lellis and
L. Sz\'ekelyhidi for useful observations.


\begin{thebibliography}{10}

\bibitem{BLL} A. C. Bronzi, M. C. Lopes Filho and H. J. Nussenzveig Lopes,
{\it Computational visualization of Shnirelman's compactly
supported weak solution}, Physica D  237, 1989--1992, 2008.


\bibitem{B} T. Buckmaster, \emph{Onsager's conjecture almost everywhere in time}, Preprint,  	
arXiv:1304.1049v2, 2013.

\bibitem{BLS} T. Buckmaster, C. De Lellis and L. Sz\'ekelyhidi Jr,
\emph{Transporting microstructure and dissipative Euler flows}, Preprint,
arXiv:1302.2815v2, 2013.


\bibitem{CLS} A. Choffrut, C. De Lellis and L. Sz\'ekelyhidi Jr,
\emph{Dissipative continuous Euler flows in two and three dimensions}, Preprint,
arXiv:1205.1226v1, 2012.

\bibitem{CFG} D. Cordoba, D. Faraco and J. Gancedo, \emph{Lack of uniqueness
for weak solutions of the incompressible porous media
equation}, Arch. Ration. Mech. Anal.  200,  no. 3, 725--746, 2011.

\bibitem{Isett13} P. Isett, \emph{H\"older Continuous Euler Flows in Three Dimensions with
Compact Support in Time}, Preprint, arXiv:1211.4065v3, 2013.

\bibitem{delellis13}C. De Lellis and L. Sz\'ekelyhidi, \emph{Dissipative continuous Euler flows},
Invent. Math. 193, no. 2, 377--407, 2013.

\bibitem{delellis12}C. De Lellis and L. Sz\'ekelyhidi, \emph{
The h-principle and the equations of fluid dynamics}, Bull. Amer. Math. Soc. (N.S.) 49, no. 3, 347--375, 2012.

\bibitem{delellis2}C. De Lellis and L. Sz\'ekelyhidi, \emph{On admissibility
criteria for weak solutions of the Euler equations}, Arch. Ration. Mech. Anal.
195, 225-260, 2010.

\bibitem{delellis}C. De Lellis and L. Sz\'ekelyhidi, \emph{The Euler equation as
a  differential inclusion}, Ann. of Math. (2)  170,  no. 3, 1417--1436, 2009.




\bibitem{shvydkoy} R. Shvydkoy, \emph{Convex integration for a class of active
scalar equations}, J. Amer. Math. Soc.  24,  no. 4, 1159--1174, 2011.


\end{thebibliography}
\end{document}